\documentclass[a4paper]{amsart}

\usepackage[english]{babel}
\usepackage[latin1]{inputenc}
\usepackage{enumerate} 
\usepackage{latexsym} 
\usepackage{amsthm} 
\usepackage{amssymb} 
\usepackage{amsmath} 
\usepackage{verbatim} 
\usepackage{yfonts}
\usepackage[all,arc,curve,color,frame]{xy}
\usepackage{etoolbox} 
\usepackage[pdftex]{graphicx}	
\usepackage{multirow} 
\usepackage{tikz}
		
\newtheorem{theorem}{Theorem}[section]
\newtheorem{lemma}[theorem]{Lemma}
\newtheorem{corollary}[theorem]{Corollary}
\newtheorem{proposition}[theorem]{Proposition}
\newtheorem{example}[theorem]{Example}

\newtheorem{question}[theorem]{Question}
\theoremstyle{definition}
\newtheorem{definition}[theorem]{Definition}

\newtheorem*{assumption*}{Standing Assumption}
\theoremstyle{remark}
\newtheorem{remark}[theorem]{Remark}

\providecommand{\scr}{\mathcal} 
	
\def\relay#1#2{%
  \expandafter\def\csname #1\endcsname{#2}%
}

\def\makecal#1{%
	\relay{c#1}{\ensuremath{\mathcal{#1}}}%
}
\newcommand{\makebb}[1]{\relay{bb#1}{\ensuremath{\mathbb{#1}}}}

\forcsvlist{\makecal}{X,Y,K,N,R,F,Q,P,U,O,M,E}
\forcsvlist{\makebb}{R,N,C,Q,D,Z,F,T}

\newcommand{\makemathop}[1]{\expandafter\DeclareMathOperator\expandafter{\csname #1\endcsname}{#1}}

\forcsvlist{\makemathop}{id, coker, ev, End}

\def\newspan{\operatorname{span}}

\newcommand{\N}{\mathbb{N}}

\newcommand{\Z}{\mathbb{Z}}

\newcommand{\TT}{\mathcal{T}}

\newcommand{\OO}{\mathcal{O}}

\numberwithin{equation}{section}

\title[Leavitt $R$-algebras over countable graphs embed into $L_{2,R}$]{Leavitt $R$-algebras over countable graphs \\ embed into $L_{2,R}$}

\author{Nathan Brownlowe}
\address[Nathan Brownlowe]{School of Mathematics and Applied Statistics, University of 
Wollongong, Australia}
\email{nathanb@uow.edu.au}

\author{Adam P W S{\o}rensen}
\address[Adam P W S{\o}rensen]{Department of Mathematics, University of Oslo, Norway, and School of Mathematics and Applied Statistics, University of Wollongong, Australia}
\email{apws@math.uio.no}


\date{}

\begin{document}

\begin{abstract}
For a commutative ring $R$ with unit we show that the Leavitt path algebra $L_R(E)$ of a graph $E$ embeds into $L_{2,R}$ precisely when $E$ is countable. Before proving this result we prove a generalised Cuntz-Krieger Uniqueness Theorem for Leavitt path algebras over $R$.
\end{abstract}

\maketitle

\section{Introduction}

In \cite{LeavittOriginal} Leavitt introduced a class of rings $R$, which fail to possess the Invariant Basis Number Property in a strong sense. 
A ring $R$ does not have invariant basis number if there exist $m,n\in\N$, $m\not=n$, with $R^m\cong 
R^n$ as modules; such a ring is said to be of module type $(m,n)$ if $m$ is the smallest natural number such that $R^m \cong R^k$ for some $k > m$ and if $n>m$ is the minimal natural number with $R^m \cong R^n$. 
For $K$ a field and $m,n$ any pair of natural numbers with $1 \le m<n$, Leavitt introduced what is now called the Leavitt algebra $L_K(m,n)$, which is an algebra over $K$ of module type $(m,n)$.
Furthermore the Leavitt algebras are universal for their module type. 

In \cite{CuntzOn} Cuntz introduced a class of $C^*$-algebras now referred to as Cuntz algebras. 
While it was not realised at the time the Cuntz algebras $\cO_n$ are natural $C^*$-algebraic analogues of the Leavitt algebras $L_K(1,n)$ --- they are universal for the same defining relations (in each their category) and $L_{\bbC}(1,n)$ is dense in $\cO_n$. 
Partly for this reason we denote $L_K(1,n)$ by $L_{n,K}$ from here on out. 
Cuntz algebras have by now been generalised in many different ways, including by the Cuntz-Krieger algebras of \cite{CuntzKrieger} and the more general $C^*$-algebras associated to directed graphs \cite{EnomotoWatatani,KumjianPaskRaeburnRenault}.
Graph $C^*$-algebras have become their own industry in $C^*$-algebra theory, and have recently been used in the ongoing classification of $C^*$-algebras program \cite{EilersRestorffRuizSubquotients,RuizSimsSorensonUCTKDimOne}.
Inspired by the success of graph $C^*$-algebras, Abrams and Aranda Pino \cite{AbramsPinoOriginalLeavitt} and, independently, Ara, Moreno and Pardo \cite{AraMorenoPardoKTheoryGraphAlgberas} generalised Leavitt algebras by introducing Leavitt path algebras associated 
to directed graphs. 
Since the appearance of \cite{AbramsPinoOriginalLeavitt,AraMorenoPardoKTheoryGraphAlgberas} Leavitt path algebras have enjoyed a similar amount of attention as their $C^*$-algebraic cousins, including Tomforde's work on Leavitt path algebras with coefficients in arbitrary (unital) commutative rings \cite{TomfordeLeavittOverRing}.
For an account of the history of Leavitt path algebras see the excellent survey paper \cite{AbramsTheFirstDecade}.

One of the main trends in Leavitt path algebras has been a desire to get a purely algebraic version of the celebrated classification theorem of Kirchberg and Phillips.
The Kirchberg-Phillips Theorem (see \cite{KirchbergClassification, PhillipsClassification} and \cite{KirchbergPhillipsEmbedding}) shows that a class of purely infinite $C^*$-algebras are completely classified by their $K$-theory.  
In a precursor to his proof of this classification theorem, Kirchberg proved three seminal theorems, which have been labelled ``Kirchberg's Geneva Theorems'' (see MR1780426):
\begin{enumerate}
	\item A separable $C^*$-algebra is exact if and only if it is a 
	sub-$C^*$-algebra of the Cuntz algebra $\cO_2$;
	\item $A \otimes \cO_2$ is isomorphic to $\cO_2$ if and only if $A$ is a 
	unital, simple, separable and nuclear $C^*$-algebra; and 
	\item if $A$ is a separable, simple, nuclear $C^*$-algebra, then $A$ is isomorphic 
	to $A \otimes \cO_\infty$ if and only if $A$ is purely infinite.
\end{enumerate}

While there are many similarities between the theories of graph 
$C^*$-algebras and Leavitt path algebras, the algebraic analogues of the Geneva Theorems have provided some of the most striking {\em differences} between them so far. 
Ara and Corti{\~n}as studied the Hochschild homology of tensor products of Leavitt path algebras in \cite{AraCortinasTensorProducts}, and concluded that the natural algebraic version of Geneva Theorems (2) and (3) do not hold. 
In particular, Ara and Corti{\~n}as show that for $K$ any field $L_{2,K} \otimes L_{2,K}$ is not\footnote{This is in contrast to $\cO_2\otimes\cO_2\cong\cO_2$, which was used in the proof of Theorems (1) and (2), but was published by R{\o}rdam in \cite{RordamO2} (and proved earlier by Elliott).} Morita equivalent to $L_{2,K}$, thus providing a counterexample to (2). 
In fact they show that $L_{2,K} \otimes L_{2,K}$ cannot be isomorphic to any Leavitt path algebra. 
They also show that $L_{\infty,K}\otimes L_K(E)$ is not Morita equivalent to $L_{K}(E)$ for a class of graphs $E$, thus providing a counterexample to (3). 
The result $L_{2,K}\otimes L_{2,K}\not\cong L_{2,K}$ was also proved independently by Dicks and by Bell and Bergman. 
Both proofs are unpublished but are described by Abrams in \cite[Section~3.5]{AbramsTheFirstDecade}.

In this paper we commence an investigation into a possible algebraic analogue of Geneva Theorem (1).
Rather than attempt a complete generalization, we take the more modest approach of investigate whether all Leavitt path algebras over $R$ embed into $L_{2,R}$.
Our main result is that the Leavitt path algebra $L_R(E)$ of any countable graph $E$ does indeed embed into $L_{2,R}$.
It is not hard to see that countability is a necessary assumption: a result of \cite{AbramsRangaswamy} implies that 
$L_{2,R}$ has at most countable dimension and Leavitt path algebras of uncountable graphs contain uncountable linearly independent subsets, which means they cannot possibly embed into $L_{2,R}$.

To establish the injectivity of the embeddings $L_R(E)\hookrightarrow L_{2,R}$ we have proved a generalised Cuntz-Krieger Uniqueness Theorem for arbitrary graphs, in the spirit of Szymanski's Unqueness Theorem for graph $C^*$-algebras \cite[Theorem~1.2]{SzymanskiUniqueness}. 
Our Uniqueness Theorem generalises the Cuntz-Krieger Uniqueness Theorem proved by Tomforde \cite[Theorem~6.5]{TomfordeLeavittOverRing} for graphs in which every cycle has an exit (Condition (L)).

\section{Background}\label{sec: background}

%
%

A {\em directed graph} $E=(E^0,E^1,r,s)$ consists of sets $E^0$ and 
$E^1$, and maps $r,s \colon E^1\to E^0$. We call a graph $E$ {\em countable} if $E^0$ 
and $E^1$ are countable, and {\em uncountable} otherwise. We call elements of 
$E^0$ vertices, and 
elements of $E^1$ edges. The maps $r$ and $s$ are called the range and source 
maps, respectively. A vertex $v\in E^0$ is called a {\em sink} if 
$s^{-1}(v)=\varnothing$, and is called an {\em infinite emitter} if 
$|s^{-1}(v)|=\infty$. 

A {\em path} in a graph $E$ is a sequence of edges $\alpha=e_1e_2\dots e_n$ 
with $r(e_i)=s(e_{i+1})$ for all $1\le i\le n-1$. A path of 
$n$ edges is said to have {\em length} $n$. The set of paths of length $n$ is 
denoted $E^n$, and the set of all finite paths is denoted 
$E^*=\cup_{n=0}^\infty E^n$. The range and source maps extend 
to finite paths in the obvious way. A {\em cycle} $\alpha=e_1\dots e_n$ is a 
path of length at least one which satisfies $s(\alpha)=r(\alpha)$. Cycles are also called closed paths in the literature. A cycle is called {\em vertex simple} if no vertices are repeated. An {\em 
exit} for a path $\alpha=e_1\dots e_n$ is an edge $e$ such that $s(e)=s(e_i)$ 
and $e\not= e_i$ for some $1\le i\le n$. A graph $E$ satisfies Condition (L) 
if every cycle has an exit. 

An {\em infinite path} in a graph $E$ is an infinite sequence of edges 
$e_1e_2\dots$ with $r(e_i)=s(e_{i+1})$ for all $i\ge 1$. We denote the set of 
infinite paths by $E^\infty$.

\begin{definition}\label{def: ghost edges}
Given a graph $E$ we define $(E^1)^*$ to be the set of formal symbols 
$\{e^*:e\in E^1\}$. In the literature, elements $e^*$ are called ghost edges, 
and elements of $E^1$ are called real edges. We also require the set of formal 
symbols $\{v^*:v\in E^0\}$, but we insist that $v^*=v$. For $\alpha=e_1\dots 
e_n$ we define $\alpha^*=e_n^*\dots e_1^*$. 
\end{definition}  

The following definition first appeared in \cite{AbramsPinoArbritraryGraphs} 
(and appeared earlier in 
\cite{AbramsPinoOriginalLeavitt} in a less general setting).

\begin{definition}\label{def: Leavitt path algebra over field}
Let $E$ be a directed graph and $K$ a field. The {\em Leavitt path algebra of 
$E$ with coefficients in $K$}, denoted 
$L_K(E)$, is the universal $K$-algebra generated by a family $\{v,e,e^*:v\in 
E^0,e\in E^1\}$ satisfying
\begin{enumerate}
\item[(1)] $\{v:v\in E^0\}$ consists of pairwise orthogonal idempotents;
\item[(2)] $s(e)e=er(e)=e$ for all $e\in E^1$;
\item[(3)] $r(e)e^*=e^*s(e)=e^*$ for all $e\in E^1$;
\item[(4)] $e^*f=\delta_{e,f}r(e)$ for all $e,f\in E^1$; and
\item[(5)] $v=\sum_{s(f)=v}ff^*$ for all $v\in E^0$ with 
$0<|s^{-1}(v)|<\infty$.  
\end{enumerate}
If $S$ is a ring and $\{v,e,e^*:v\in E^0,\, e\in E^1\}\subseteq S$ is a 
collection satisfying (1)--(5) above, we call this collection a {\em Leavitt 
$E$-family} in $S$.
\end{definition}

In \cite[Definition~3.1]{TomfordeLeavittOverRing} the definition of a Leavitt 
path algebra was 
generalised to include more general coefficients. For $E$ a directed graph and 
$R$ a commutative ring with unit, the {\em Leavitt path algebra with 
coefficients in $R$}, denoted $L_R(E)$, is the universal $R$-algebra generated 
by a Leavitt $E$-family. Recall from 
\cite[Proposition~3.4]{TomfordeLeavittOverRing} that the 
elements $\{v,e,e^*:v\in E^0,\, e\in E^1\}$ of the Leavitt $E$-family 
generating $L_R(E)$ are all nonzero. Moreover, that
\[
L_R(E) = \newspan_R\{\alpha\beta^*: \alpha,\beta\in E^*\text{ with } 
r(\alpha)=r(\beta)\},
\]
and $rv\not=0$ for all $v\in E^0$ and $0\not=r\in R$. Also recall that by 
\cite[Proposition 4.9]{TomfordeLeavittOverRing} the set $E^*$ and the set of 
all ghost paths are both linearly independent sets. 

The map $\alpha\beta^*\mapsto \beta\alpha^*$ extends to an $R$-linear 
involution of $L_R(E)$. 

\begin{assumption*}
Throughout this paper $R$ will always be a commutative ring with unit.
\end{assumption*}

We will be focussed on the graph
\vspace{-0.3cm}

\begin{equation}\label{eq: 2 loop graph}
\begin{tikzpicture}[baseline=(current  bounding  box.center)]

\node[circle, draw=black,fill=black, inner sep=1pt] (u) at (0,0) {};
    
\draw[-latex,thick] (u) .. controls +(225:1.5cm) and +(135:1.5cm) .. (u);
\draw[-latex,thick] (u) .. controls +(315:1.5cm) and +(45:1.5cm) .. (u);
        
\node at (-1,0) {$a$};
\node at (1,0) {$b$};

\end{tikzpicture}
\end{equation}
\vspace{-0.3cm}

\noindent and its Leavitt path algebra $L_{2,R}$, which is the universal 
$R$-algebra generated by elements $a,a^*,b,b^*$ satisfying 

\begin{equation}\label{eq: L2 relations}
a^*a=b^*b=1=aa^*+bb^*. 
\end{equation}
From these relations we see that $aa^*a=a$ and $b^*bb^*=b^*$, and hence $b^*a=b^*(aa^*+bb^*)a=b^*a+b^*a$. So $b^*a=0$, and a similar argument gives $a^*b=0$. Hence the universal $R$-algebra generated by $a$ and $b$ satisfying (\ref{eq: L2 relations}) is indeed the Leavitt path algebra for the above graph.

Note that when $R$ is a field $K$, $L_{2,K}$ is the 
Leavitt
algebra of module type $(1, 2)$.

\begin{definition}\label{def: projs}
Let $A$ be an involutive $R$-algebra. 
We call an element $u\in A$ a unitary if $u^*u=uu^*=1$. 
We call an element $p \in A$ a projection if $p = p^2 = p^*$.
We say that two projections are Murray-von Neumann equivalent, written $p_1 \sim p_2$, if there 
exist $x$ such that $p_1 = xx^*$ and $p_2 = x^*x$. 
If two projections $p_1,p_2$ are Murray-von Neumann equivalent, we can choose a witness $x$ such 
that $x x^* x = x$ 
(see \cite[Proposition 4.2.2]{BlackadarKTheoryBook}).
\end{definition}

\begin{remark}\label{rem: proj equiv}
The equivalence relation introduced above is often referred to as 
$*$-equivalence to distinguish it from the equivalence relation on idempotents 
given by letting $e,f$ be equivalent if there exist $x,y$ such that $e = xy$ 
and $f = yx$.
In general these notions of equivalence do not coincide, for instance they do not coincide on $M_2(\bbQ)$, see \cite[Theorem 1.12]{AraMatrixRings} for details and a much more general statement.
The authors thanks Pere Ara for pointing out the difference to them.  
\end{remark}

\begin{definition}\label{def: partial unitary}
Let $R$ be a commutative ring with unit and let $A$ be an involutive 
$R$-algebra. 
A {\em partial unitary} is an element $u \in A$ with 
$u^*u=uu^*=p$ for some projection $p\in A$.
\end{definition}

The subalgebra of $A$ generated by the partial unitary $u$ and $p$ is 
a $*$-homomorphic image of $L_R[z,z^{-1}]$, the Laurent polynomials over $R$. 
This allows one to apply any Laurent polynomial to $u$.
More elaborately: 

\begin{remark}\label{rmk: functional calculus}
$L_R[z,z^{-1}]$ is the universal $R$-algebra generated by the partial unitary 
$z$; that is, if $u \in A$ is a partial unitary there exists a $*$-homomorphism 
$\phi \colon L_R[z,z^{-1}] \to A$ such that 
\[
  \phi(z) = u, \quad \phi(z^{-1}) = u^*, \quad \text{and} \quad \phi(1) = p. 
\]
Let $q$ be a polynomial in $L_R[z,z^{-1}]$ and write 
\[
  q(z) = \sum_{n=-m}^m k_n z^{n}  
\]
where each $k_n \in R$. 
Then 
\[
  \phi(q) = \sum_{n=-m}^{-1} k_n (u^*)^{-n} + p + \sum_{n=1}^m k_n 
  u^{n}.  
\]
In our notation we will omit the map $\phi$ and simply write $q(u)$. 
\end{remark}

\begin{definition}\label{def: full spectrum}
Let $R$ be a commutative ring with unit and let $A$ be an involutive 
$R$-algebra. We say a partial unitary $u\in A$ has {\em full spectrum}\footnote{This term comes from the notion of full spectrum in $C^*$-algebra theory.} if 
$q(u) \neq 0$ for all nonzero $q \in R[x]$. 
\end{definition}

Note that if $u$ has full spectrum, then $q(u)\not= 0$ for all nonzero Laurent polynomials $q$.

\begin{example}\label{eg: cycle with no exit}
The most basic example of a full spectrum partial unitary in a Leavitt path 
algebra is a cycle $\alpha$ without an exit. This follows because the set 
of paths of real edges is linearly independent 
\cite[Proposition~4.9]{TomfordeLeavittOverRing}, and hence $q(\alpha) \neq 0$ 
for all non-zero $q \in R[x]$. 
\end{example}

\section{A generalised Cuntz-Krieger Uniqueness Theorem}\label{sec: gen CKUT}

We now prove a uniqueness theorem for Leavitt path algebras with coefficients 
in a unital, commutative ring $R$. This result generalises the uniqueness theorem of 
Tomforde \cite[Theorem~6.5]{TomfordeLeavittOverRing}, in which graphs are 
assumed to satisfy 
Condition (L), and is the algebraic analogue of Szymanski's uniqueness theorem 
for graph $C^*$-algebras \cite[Theorem~1.2]{SzymanskiUniqueness}. (See also Nagy and Reznikoff's uniquess theorem in \cite{NagyReznikoffAbelianCores}.)

\begin{theorem}[Generalised Cuntz-Krieger Uniqueness Theorem]\label{thm: Gen 
CKUT}
Let $E$ be a graph, $R$ a commutative ring with unit, and $\phi$ a ring 
homomorphism of $L_R(E)$ into a ring $S$.
Then $\phi$ is injective if and only if the following conditions are satisfied:
\begin{enumerate}
\item[(G1)] $\phi(rv) \neq 0$ for all $v\in E^0$ and all $r \in 
R\setminus\{0\}$; and
\item[(G2)]  $\phi(q(\alpha)) \neq 0$ for all cycles $\alpha$ without an exit, 
and all 
nonzero polynomials $q \in R[x]$.
\end{enumerate} 
\end{theorem}

A version of this Theorem for fields can be derived either from \cite[Proposition~2]{RangaswamyGeneratorsOfIdeals} or from \cite[Proposition~3.1]{PinoMartinMartinSilesSocle}. 
We work over rings rather than fields, so must do some extra work. 

To prove Theorem~\ref{thm: Gen 
CKUT} we study the ideal structure of Leavitt path algebras over a commutative ring.
In Section 7 of \cite{TomfordeLeavittOverRing} Tomforde carries out an investigation into so called basic ideals, which are ideals $I$ with the property that if $rv \in I$ for some nonzero $r$ then $v \in I$. 
Here we will be interested in a different class of ideals, namely those that contain no elements of the form $rv$.  

\begin{definition}\label{def: vertex free}
An ideal $I \subseteq L_R(E)$ is called \emph{vertex free} if for $r \in R$ 
and $v \in E^0$ we have
\[
	r v \in I \implies r = 0.
\]
\end{definition}

We now prove some structure results for vertex free ideals. 
We follow the proof of the Cuntz-Krieger Uniqueness Theorem \cite[Theorem~6.5]{TomfordeLeavittOverRing}, but we keep track of what happens at cycles without exits. 
Our goal is to show that, in a certain sense, vertex free ideals come from cycles without exits, see Proposition \ref{prop:loopelements} for a precise statement.  

In the following two lemmas $E$ is a graph, $R$ is a commutative ring with unit and we work in $L_R(E)$. 

\begin{lemma} \label{lem:elmentinvertexfree}
For every nonzero vertex free ideal $I$ there exists a vertex $u \in 
E^0$, cycles $\beta_1, \beta_2, \ldots, \beta_n$ based at $u$,  and elements $r_0, r_1, 
\ldots, r_n \in R\setminus\{0\}$ such that 
\[
	0\not=r_0 u + \sum_{i=1}^n r_i \beta_i \in I.
\] 
\end{lemma}
\begin{proof}
Since $I$ is nonzero, there is a nonzero $x \in I$.
By \cite[Lemma 6.4]{TomfordeLeavittOverRing} we can assume that $x$ is a 
polynomial in real edges. 
As $E^0$ form a set of local units for $L_R(E)$, we can find some $u \in E^0$ 
with $y = ux \neq 0$.
Since $I$ is vertex free, $y$ has the form 
\[
	y = s_0 u + \sum_{j=1}^m s_j \gamma_j,
\] 
for $s_j \in R$, nonzero for $j \neq 0$, and $\gamma_j$ distinct paths of 
positive length with source $u$.

We now claim that there exists a nonzero $z \in I$ of the form 
\begin{eqnarray}\label{eqn: z form}
	z = t_0 u + \sum_{l=1}^k t_l \beta_l,
\end{eqnarray}
where $t_0, t_1, \ldots, t_k$ are all nonzero, and $\beta_l$ distinct paths 
of positive length with source $u$. If $s_0 \neq 0$, then $z = y$ does the 
job.  
If $s_0 = 0$, we relabel the terms of $y$ so that $|\gamma_1| \leq |\gamma_2| 
\leq \cdots \leq |\gamma_m|$, and consider $\gamma_1^* y$. 
Then 
\[
	\gamma_1^* y = s_1 r(\gamma_1) + \sum_{j = 2}^m s_j \gamma_1^* 
\gamma_j.
\]
Since $I$ is vertex free, at least one of the $\gamma_1^* \gamma_j$ is 
nonzero, and since $\gamma_1$ has minimal length among the $\gamma_j$, 
whenever $\gamma_1^* \gamma_j$ is nonzero it is a path of positive length 
with source $r(\gamma_1)$. 
Thus 
\[
\gamma_1^* y = s_1 r(\gamma_1) + \sum_{l = 2}^n r_l \beta_l
\]  
for some non-zero $r_l \in R$ and distinct paths $\beta_l$ with source $r(\gamma_1)$.
By \cite[Proposition 4.9]{TomfordeLeavittOverRing} $E^*$ is linearly 
independent in $L_R(E)$, so since $s_1 \neq 0$ and since all the $\beta_l$ have 
positive length we have $\gamma_1^*y \neq 0$. Hence $z=\gamma_1^*y$ has the 
desired form. 

Now, fix $z$ of the form (\ref{eqn: z form}) and consider $w = z u \in I$. We have 
\[
  w = t_0 u + \sum_{\{l : s(\beta_l) = u = r(\beta_l)\}} t_l \beta_l.
\]
Since $t_0 \neq 0$ and $I$ is vertex free, the set $\{l : s(\beta_l) = u = 
r(\beta_l)\}$ is nonempty and so $w$ has the desired form. 
As we noted in our construction of $z$, since $t_0 \neq 0$ and  
all the $\beta_l$ have positive length, we have $w \neq 0$.
\end{proof}

\begin{lemma} \label{lem:killingpath}
Let $\beta_1, \beta_2, \ldots, \beta_n$ be distinct cycles all based at some 
vertex $u \in E^0$.
If cycles based at $u$ have exits, then there exists $\gamma \in 
E^*$ 
with $s(\gamma) = u$ and such that $\gamma^* \beta_i \gamma = 0$ for all $i = 
1,2,\ldots,n$.
\end{lemma}

Note that either all cycles based at a vertex $u$ have 
exits, or no cycle based at $u$ has exits, which explains the wording of the 
above lemma.
In the latter case there is a cycle $\alpha$ of minimal length based at $u$ and all other cycles are of the form $\alpha^n$. 

\begin{proof}[Proof of Lemma~\ref{lem:killingpath}]
Pick a cycle $\tau$ based at $u$ of minimal length. 
Since $\tau$ has an exit we can write $\tau = \mu \nu$ for $\mu \in 
E^*, \nu \in E^* \setminus E^0$ with $|s^{-1}(r(\mu))| \ge 2$.
Let $f \in s^{-1}(r(\mu))$ be distinct from the first edge in $\nu$. 
Fix an $m \in \N$ such that $\tau^m$ is longer than all the $\beta_i$, and put
\[
	\gamma = \tau^m \mu f.
\]	
To compute $\gamma^* \beta_i \gamma$ we look at two cases. First suppose that 
$\beta_i$ is an initial segment of $\tau^m$. Since $\tau$ is a cycle of minimal 
length and $\beta_i$ is a cycle we have 
$\beta_i = \tau^k$ for some $1 \leq k < m$.
Hence
\[
	\gamma^* \beta_i \gamma = f^* \mu^* (\tau^*)^m \tau^k \tau^m \mu f = 
f^* \mu^* \tau^k \mu f = f^* \mu^* \mu \nu \tau^{k-1} \mu f 
= f^* \nu \tau^{k-1} \mu f = 0,
\]	
since $f$ is distinct from the first edge in $\nu$.

Now suppose $\beta_i$ is not an initial segment of $\tau^m$. Then $(\tau^m)^* 
\beta_i = 0$, and hence $\gamma^* \beta_i \gamma = 0$.
\end{proof}

\begin{proposition} \label{prop:loopelements}
Let $E$ be a graph and $R$ a commutative ring with unit. 
If $I \subseteq L_R(E)$ is a non-zero vertex free ideal, then $I$ contains a 
nonzero 
element of the form
\[
	s_0 u + \sum_{i=1}^n s_1 \alpha^i,
\]
for some vertex $u \in E^0$, a cycle with no exits $\alpha$ based at $u$, and 
$s_0, s_1, \ldots, s_n \in R$ not all zero. 
\end{proposition}
\begin{proof}
Since $I$ is vertex free we can apply Lemma~\ref{lem:elmentinvertexfree} to 
find a nonzero $w \in I$ of the form 
\[
	w = r_0 u + \sum_{i=1}^n r_i \beta_i,
\]
for some vertex $u \in E^0$, cycles $\beta_1, \beta_2, \ldots, \beta_n$ based 
at $u$, and $r_0, r_1, \ldots, r_n \in R\setminus\{0\}$.
If cycles based at $u$ had exits, we could apply Lemma \ref{lem:killingpath} 
to find a path 
$\gamma \in E^*$ with $s(\gamma) = u$ and $\gamma^* \beta_i \gamma = 0$ for 
all $i = 1,2,\ldots, n$. 
Hence 
\[
	\gamma^* w \gamma = r_0 \gamma^* u \gamma + \sum_{i=1}^n r_i \gamma^* 
\beta_i \gamma = r_0 r(\gamma),
\]
which contradicts that $I$ is vertex free. We conclude that cycles based at $u$ 
do not have exits. Since this is the case, there is a cycle $\alpha$ based at 
$u$ of minimal length.
Furthermore all the $\beta_i$ must have the form $\beta_i = \alpha^{k_i}$ for 
some $k_i \in \N$.
So $w$ has the desired form (by padding the sum with zero elements if 
necessary).
\end{proof}

\begin{proof}[Proof of Theorem~\ref{thm: Gen CKUT}]
The ``only if'' direction is immediate. For the ``if'' direction suppose that 
$\phi\colon L_R(E)\to S$ satisfies (G1) and (G2). Suppose for contradiction 
that $I=\ker\phi$ is nonzero. We know from (G1) that $I$ is a vertex free 
ideal. Then by Proposition \ref{prop:loopelements} there is a 
nonzero element $x \in I$ of the form 
\[
x=	s_0 u + \sum_{i=1}^n s_i \alpha^i,
\]
for some vertex $u \in E^0$, a cycle with no exits $\alpha$ based at $u$, and 
elements $s_0, s_1, \ldots, s_n$ in $R$ not all zero. This means the nonzero 
polynomial $q(x)=\sum_{i=0}^ns_ix^i$ in $R[x]$ satisfies 
\[
 \phi(q(\alpha))=0,
\]
which contradicts (G2). Hence we must have $I=\{0\}$, and $\phi$ is injective. 
\end{proof}            

We can rephrase the Generalised Cuntz-Krieger Uniqueness Theorem for 
$*$-homomorphisms into involutive $R$-algebras. It is this result we apply in 
the next section. 

\begin{corollary}\label{cor: Gen CKUT StarAlg}
Let $E$ be a graph, $R$ a commutative ring with unit, and $\phi$ a 
$*$-homomorphism of $L_R(E)$ into an involutive $R$-algebra $A$.
Then $\phi$ is injective if and only if the following conditions are satisfied:
\begin{enumerate}
\item[(1)] $\phi(rv) \neq 0$ for all $v\in E^0$ and all $r \in R \setminus 
\{0\}$; and
\item[(2)] $\phi(\alpha)$ has full spectrum in $A$ for every cycle $\alpha$ 
without an exit.
\end{enumerate} 
\end{corollary}

\begin{proof}
Since for any cycle $\alpha$ without an exit, and for any $q(x)\in R[x]$, we 
have $\phi(q(\alpha))=q(\phi(\alpha))$, it follows that (2) is equivalent to (G2). 
The result now follows from Theorem~\ref{thm: Gen CKUT}.
\end{proof}

\section{Embedding $L_R(E)$ into $L_{2,R}$}\label{sec: embedding L(E) in L2}

The main result of this paper is the following.

\begin{theorem}\label{thm: Embed L(E) into L2}
Let $R$ be a commutative ring with unit and let $E$ be a directed graph. 
If $E$ is countable, then there is a $*$-algebraic embedding of the Leavitt path algebra $L_R(E)$ into $L_{2,R}$, and if $E^0$ is finite, then this embedding can be chosen to be unital. 
If $E$ is uncountable then there is no embedding of $L_R(E)$ into $L_{2,R}$ as $R$-modules. 
\end{theorem} 

As discussed in the introduction, a dimension argument shows that $L_{R}(E)$ cannot embed into $L_{2,R}$ for uncountable $E$.

The first step in proving that the countability of $E$ is sufficient for such an embedding is to show that the unit in $L_{2,R}$ may be broken into countably many pairwise orthogonal projections that are all Murray-von Neumann equivalent to the unit. 
This will be done in Proposition \ref{prop: projs in L2}, but first we prove a lemma.

\begin{lemma} \label{lem: some projections are nice}
Let $R$ be a commutative ring with unit and let $p \in L_{2,R}$ be a projection. 
If $p \sim 1$, in the sense of Definition~\ref{def: projs}, then 
\begin{enumerate}
	\item[(i)] $r p \neq 0$ for all $r \in R \setminus \{0\}$ and
	\item[(ii)] $p L_{2,R} p \cong L_{2,R}$ as $*$-algebras. 
\end{enumerate} 
\end{lemma}
\begin{proof}
Since $p \sim 1$ we can find a $t \in L_{2,R}$ such that $t^*t = 1$ and $tt^* = p$.
For (i) suppose that $r p = 0$ for some $r \in R$.
Then 
\[
	0 = t^* 0 t = t^* r p t = r t^* t t^* t = r 1.
\]
By \cite[Proposition~4.9]{TomfordeLeavittOverRing} this implies that $r = 0$.

The map $\phi \colon L_{2,R} \to L_{2,R}$ given by $\phi(x) = t x t^*$ is a $*$-homomorphism. 
Since $\phi(1) = p$ it follows from (i) and the Cuntz-Krieger Uniqueness Theorem (\cite[Theorem~6.5]{TomfordeLeavittOverRing}) that $\phi$ is injective. 
As $p t = t$ and $t^* p = t^*$ we have that the image of $\phi$ is contained in $p L_{2,R} p$, and as $\phi(t^* x t) = p x p$ for all $x \in L_{2,R}$, the image of $\phi$ is all of $p L_{2,R} p$.
Hence $\phi$ shows that $L_{2,R} \cong p L_{2,R} p$ as $*$-algebras, i.e. (ii) holds.
\end{proof}

\begin{proposition}\label{prop: projs in L2}
Let $R$ be a commutative ring with unit. 
For every nonempty subset $I \subseteq \bbN$ there exist pairwise orthogonal nonzero projections $\{p_i : i \in I\} \subset L_{2,R}$ such that 
\begin{enumerate}
\item[(i)] $p_i \sim 1$, in the sense of Definition~\ref{def: projs}, for all $i$;

\item[(ii)] $rp_i\not=0$ for all $i$ and for all $r\in R\setminus\{0\}$; 

\item[(iii)] $p_i L_{2,R} p_i \cong L_{2,R}$ for all $i$; and

\item[(iv)] if $I$ is finite, then $\{p_i : i \in I\}$ can be chosen so that $\sum_{i \in I} p_i = 1$ also holds.
\end{enumerate}
\end{proposition}

\begin{proof}
For each $i \in I$ we define 
\[
	t_i = a^i b	\quad \text{ and } \quad p_i = t_i t_i^*.
\]
Since $t_i^* t_i = b^* (a^i)^* a^{i} b = 1$ we have that $p_i \sim 1$.
It follows from Lemma~\ref{lem: some projections are nice} that the $p_i$ satisfy (ii) and (iii).
For $i > j$ we have $t_i^*t_j = b^* (a^{i-j})^* b = 0$, and for $i < j$ we have $t_i^*t_j = b^* a^{j-i} b = 0$. 
Hence the $p_i$ are pairwise orthogonal.

For (iv) we may assume that $I=\{1,\dots,n\}$.
We will prove by induction on $n$ that there exist orthogonal projections 
$p_i$ satisfying (i)--(iv).

If $n = 1$ then we simply put $p_1 = 1$.
Suppose now that our inductive hypotheses holds for all $k \leq n-1$ and pick projections $q_1, q_2, \ldots, q_{n-1}$ satisfying (i)--(iv). 
Let $t$ be such that $t^*t = 1$ and $tt^* = q_{n-1}$.
Define $p_i = q_i$ if $i < n - 1$ and 
\[
	p_{n-1} = t a (t a)^*, \quad p_{n} = t b (t b)^*. 
\]
Since $(ta)^* (t a) = a^* t^* t a = a^* a = 1$ it follows from Lemma~\ref{lem: 
some projections are nice} that $p_{n-1}$ satisfy (i)--(iii).
Similarly $p_{n}$ satisfies (i)--(iii).
We see that 
\[
	p_{n-1} p_n =  t a a^* t^* t b b^* t^* = t a a^* b b^* t^* = 0.
\]
For $i < n-1$ we have 
\[
	p_i p_{n-1} = q_i t a a^* t^* = q_i t t^* t a a^* t^* = q_i 
	q_{n-1} t a a^* t^* = 0,  
\]
and similarly $p_i p_{n} = 0$.
Thus $p_1, p_2, \ldots, p_n$ are orthogonal projections, and since 
\[
	p_{n-1} + p_n = t a a^* t^* + t b b^* t^* = t(a a^* + b 
	b^*)t^* = tt^* = q_{n-1},
\]
we have
\[
	\sum_{i=1}^n p_i = 1.
\]
Which completes the induction and therefore shows that (iv) holds. 
\end{proof}

To deal with graphs that have cycles without exits, we also need to find a full
spectrum unitary in $L_{2,R}$.

\begin{proposition}\label{prop: full spectrum}
There exists a full spectrum unitary in $L_{2,R}$.  
\end{proposition}
\begin{proof}
Define 
\[
  u = a a a^* + a b a^* b^* + b b^* b^*.
\]
A routine calculation shows that $u^*u=uu^*=1$. To see that $u$ has full spectrum, we note that $u a = a^2$,
so for any $n \geq 1$ we have
\begin{equation}\label{eq:fullspec u and a}
	u^n a = a^{n+1}.
\end{equation}
Let $q \in R[x]$ satisfy $q(u) = 0$.
We aim to show that $q$ must be the zero polynomial. 
We write 
\[
	q(x) = \sum_{n=0}^k r_n x^n,
\] 
where each $r_n \in R$. Using (\ref{eq:fullspec u and a}) we get 
\[
	0 = 0 \cdot a = q(u)\cdot a= \left( \sum_{n=0}^k r_n u^n \right) a = 
	\sum_{n=0}^{k} 
	r_n a^{n+1}.
\]
Since we know from \cite[Proposition 4.9]{TomfordeLeavittOverRing} that $\{ 
a^n : n \in \bbN \}$ 
is a linearly independent set, we have $r_n=0$ for $0\le n\le k$. So $q(x)$  
is the zero polynomial. 
Since no nonzero polynomial in $u$ is zero we conclude that $u$ has full spectrum.  
\end{proof}

\begin{remark}
How does one arrive at the particular $u$ presented in Proposition \ref{prop: full spectrum}?
The crucial point in proving that $u$ has full spectrum was that $u^2 a = a^2$. 
To get a unitary with this property, we first looked for a unital endomorphism 
$\phi$ of $L_{2,R}$ such that
$\phi(a) =  a^2$. It is well known, see e.g. \cite{CuntzAutosOfOn}, that endomorphisms of
the Cuntz algebra $\cO_2$ are in one-to-one correspondence with unitaries. Applying these ideas to Leavitt algebras 
we define $u$ as
\[
  u = \phi(a)a^* + \phi(b)b^*. 
\]
We arrived at our particular $u$ by putting $\phi(b) = aba^* + bb^*$. 
Choosing endomorphisms $\psi_n$ of $L_{2,R}$ with $\psi_n(a) = a^n$ one can construct many different full spectrum unitaries. 
\end{remark}

Proposition \ref{prop: full spectrum} lets us embed $L_R[z, z^{-1}]$ into $L_{2,R}$. 
So if we let $C_n$ denote the graph consisting of single vertex simple cycle of length $n$, then we have that $L_R(C_1)$ embeds into $L_{2,R}$. 
The next lemma gives a concrete embedding of $L_R(C_n)$ into $L_{2,R}$ for all 
$n$. 

\begin{lemma} \label{lem: embed loop}
Given pairwise orthogonal projections $p_1, p_2, \ldots, p_n$ in $L_{2,R}$ with $p_i \sim 1$ for $1 \leq i \leq n$ there exists $t_1, t_2, \ldots, t_n$ in $L_{2,R}$ such that 
\begin{enumerate}
	\item $t_i^* t_i = p_{i+1}$, for $i < n$, and $t_n^* t_n = p_1$, 
	\item $t_i t_i^* = p_i$, for all $i$, 
	\item $t_i t_i^* t_i = t_i$, for all $i$, and, 
	\item $v = t_1 t_2 \cdots t_n$ is a full spectrum partial unitary with $vv^* = p_1 = v^*v$.
\end{enumerate}
\end{lemma}
\begin{proof}
Because the projections are Murray-von Neumann equivalent, we can find $t_1, t_2, \ldots, t_{n-1} 
\in L_{2,R}$ that satisfy (1)--(3) and an $s \in L_{2,R}$ such that $ss^*	= 
p_n$ and $s^* s = 1$. 
By Proposition \ref{prop: full spectrum} there is a full spectrum unitary $\widetilde{u} \in L_{2,R}$.  
Let $u = s\widetilde{u}s^*$ and put 
\[
	t_n  = u \left( t_1 t_2 \cdots t_{n-1} \right)^*.
\]
Observe that $u^* u = p_n = uu^*$.

We check that $t_n$ satisfies (1)--(3). 
Since 
\[
	t_{i} t_{i+1} t_{i+1}^* t_{i}^* = t_i p_{i+1} t_i^* = t_i t_i^* t_i t_i^* = t_i t_i^*,
\]
for $i < n-1$, we have that 
\begin{align*}
	t_n^* t_n	&= \left( t_1 t_2 \cdots t_{n-1} \right) u^*u \left( t_1 t_2 \cdots t_{n-1} \right)^* \\
						&= t_1 t_2 \cdots t_{n-2} t_{n-1} t_{n-1}^* t_{n_2}^* \cdots t_2^* t_1^* \\
						&= t_1 t_1^* = p_1.
\end{align*}
So $t_n$ satisfies (1).
Similarly, we see that
\[
	t_{i+1}^* t_i^* t_i t_{i+1} = t_{i+1}^* p_{i+1} t_{i+1} = t_{i+1}^* t_{i+1} t_{i+1}^* t_{i+1} = t_{i+1}^* t_{i+1}, 
\]
for $i < n-1$. 
Hence
\begin{align*}
	t_n t_n^* &= u \left( t_1 t_2 \cdots t_{n-1} \right)^* \left( t_1 t_2 \cdots t_{n-1} \right) u^* \\
						&= u t_{n-1}^* t_{n_2}^* \cdots t_2^* t_1^* t_1 t_2 \cdots t_{n-1} u^* \\
						&= u t_{n-1}^* t_{n-1} u^* = u p_n u^* \\
						&= u u^* u u^* = p_n.
\end{align*}
I.e. $t_n$ satisfies (2).
Combining the above with the fact that 
\[
	p_n u = ss^* s \widetilde{u} s^* = s \widetilde{u} s^* = u,
\]
we get that 
\[
	t_n t_n^* t_n = p_n t_n = p_n u \left( t_1 t_2 \cdots t_{n-1} \right)^* = u \left( t_1 t_2 \cdots t_{n-1} \right)^* = t_n.
\]
So $t_n$ satisfies (3). 

Let $v = t_1 t_2 \cdots t_n$ and let $t = t_1 t_2 \cdots t_{n-1}$.
Then $v = (ts)\widetilde{u}(ts)^*$.
Computations like above show that
\[
	ts(ts)^* = tss^*t^* = t p_n t^* = tt^* = p_1,
\]
and that 
\[
	(ts)^*ts = s^* t^*t s = s^* p_n s = s^* s s^* s = 1.
\]
Hence 
\[
	vv^* = (ts)\widetilde{u}(ts)^* (ts)\widetilde{u}^*(ts)^* = (ts)\widetilde{u} \widetilde{u}^*(ts)^* = ts(ts)^* = p_1,
\]
and 
\[
 v^* v = (ts)\widetilde{u}^*(ts)^* ts\widetilde{u}(ts)^* = (ts)\widetilde{u}^* \widetilde{u} (ts)^* = ts(ts)^* = p_1.
\]
So $v$ is a partial unitary. 

To see that $v$ has full spectrum let $q \in R[z]$ be such that $q(v) = 0$.
Since 
\[
	v^m = ((ts)\widetilde{u}(ts)^*)^m = (ts) \widetilde{u}^m (ts)^*,
\]
and similarly 
\[
	(v^*)^m = (ts) (\widetilde{u}^*)^m (ts)^*
\]
we see that 
\[
0=	q(v) = (ts) q(\widetilde{u}) (ts)^*.
\]
An argument similar to the proof of (ii) in Lemma \ref{lem: some projections are nice} shows that conjugation by $ts$ is an isomorphism of $L_{2,R}$ onto $p_1L_{2,R}p_1$.
So we must have $q(\widetilde{u}) = 0$. 
Since $\widetilde{u}$ has full spectrum we conclude that $q$ is the zero polynomial, and therefore $v$ has full spectrum. 
\end{proof}

\begin{proof}[Proof of Theorem~\ref{thm: Embed L(E) into L2}]
Arguing as in \cite[Proposition~5 and Corollary~6]{AbramsRangaswamy}, which is for fields but works for commutative rings, one sees that $L_{2,R}$ has countable dimension as an $R$-module. 
Suppose that $E$ is uncountable, then $E^0$ is uncountable or $E^1$ is uncountable. 
If $E^0$ is uncountable then it forms an uncountable set of linearly independent elements of $L_R(E)$.
If $E^0$ is at most countable and $E^1$ is uncountable, then there is some vertex $u \in E^0$ that emits uncountably many edges.
In this case $\{ ee^* : s(e) = u \}$ is an uncountable set of orthogonal projections, so it forms an uncountable set of linearly independent elements of $L_R(E)$.
In both cases we have an uncountable collection of linearly independent elements in $L_{R}(E)$, hence there is no module embedding of $L_R(E)$ into $L_{2,R}$. 

Suppose instead that $E$ is countable. Since $E^0$ is countable 
we can use Proposition~\ref{prop: projs in L2} to get pairwise orthogonal 
nonzero projections $\{p_v : v \in E^0\} \subset L_{2,R}$ such that $p_v 
\sim 1$ and $p_v L_{2,R} p_v \cong L_{2,R}$ for all $v \in E^0$. If $E^0$ is finite, we know from (iv) of Proposition~\ref{prop: projs in L2} that these projections can be chosen to sum to the identity.  
For each $v \in E^0$ with $s^{-1}(v) \not= \varnothing$ we can, since $p_v L_{2,R} p_v \cong L_{2,R}$, use Proposition~\ref{prop: projs in L2} to choose pairwise orthogonal nonzero projections $\{ q_e^v : e \in s^{-1}(v) \} \subseteq p_v L_{2,R} p_v$ such that $q_e^v \sim p_v$, and if 
$|s^{-1}(v)|<\infty$, such that 
\[
	p_v=\sum_{e\in s^{-1}(v)}q_e^v. 
\]
Then for each $e\in E^1$ we have
\[
	q_e^{s(e)} \sim p_{s(e)} \sim 1 \sim p_{r(e)}. 
\]
If $e$ is not part of a cycle without an exit we use the definition of $\sim$ 
to choose $t_e$ such that $t_e t_e^* = q_e^{s(e)}$, $t_e^* t_e = p_{r(e)}$, 
and $t_e t_e^* t_e = t_e$. 
For each vertex simple cycle $\alpha = e_1 e_2 \cdots e_n$ without an exit, we use Lemma~\ref{lem: embed loop} to pick elements $t_{e_i}$ such that $t_{e_i} t_{e_i}^* = q_{e_i}^{s(e_{i})}$, $t_{e_i}^* t_{e_i} = p_{r(e_i)}$, $t_{e_i} t_{e_i}^* t_{e_i} = t_{e_i}$ and $t_{e_1} t_{e_2} \cdots t_{e_n}$ is a partial unitary with full spectrum.  

We claim that $\{ p_v : v \in E^0\}$ and $\{ t_e, t_e^* : e \in E^1\}$ form a 
Leavitt $E$-family. 	
By construction the $p_v$ are orthogonal projections.
Furthermore we have 
\[
	p_{s(e)} t_e = p_{s(e)} t_e t_e^* t_e = p_{s(e)} q_e^{s(e)} t_e  = q_e^{s(e)} t_e = t_e t_e^* t_e = t_e 
\]
and 
\[
	t_e p_{r(e)} = t_e t_e^* t_e = t_e.
\] 
Applying the involution we get
\[
 p_{r(e)} t_e^*	= t_e^* p_{s(e)} = t_e^*.
\]
For $e \neq f$ we have 
\[
	t_e^* t_f = t_e^* q_e^{s(e)} q_f^{s(f)} t_f = 0,
\]
and we have 
\[
	t_e^* t_e = p_{r(e)}.
\]
Finally, if $0 < s^{-1}(v) < \infty$ then 
\[
	\sum_{e \in s^{-1}(v)} t_e t_e^* = \sum_{e \in s^{-1}(v)} q_e^{s(e)} = p_{s(e)}.
\]
So $\{ p_v : v \in E^0\}$ and $\{ t_e, t_e^* : e \in E^1\}$ do indeed form a 
Leavitt $E$-family.

The universal property of $L_R(E)$ gives a $*$-homomorphism $\phi 
\colon L_R(E) \to L_{2,R}$ which takes the vertex projections in $L_R(E)$ to 
the projections $\{ p_v : v \in E^0\}$, and we know from part (iii) of 
Proposition~\ref{prop: projs in L2} that $rp_v \neq 0$ for each $v$ and for all 
$r\in R\setminus\{0\}$. Since for all cycles $\alpha$ without an exit the 
image 
$\phi(\alpha)$ is a full spectrum partial unitary by construction, it now 
follows from Corollary~\ref{cor: Gen CKUT StarAlg} that $\phi$ is injective.

Finally, in the case that $E^0$ is finite, we chose the projections $p_v$ to sum to the identity. Then $\phi(1)=\phi(\sum_{v\in E^0}v)=\sum_{v\in E^0}p_v=1$. So $\phi$ is unital, as claimed. 
\end{proof}

Theorem \ref{thm: Embed L(E) into L2} raises two natural questions.

Firstly, keeping in mind that we are trying to get an algebraic version of the first of Kirchberg's Geneva Theorems, we wonder if there are any $R$-algebras with a countable basis that do not embed into $L_{2,R}$.
In \cite{BS2} we show that $L_{2,\Z} \otimes L_{2,\Z}$ does not embed into $L_{2,\Z}$.
However, we have no counterexamples outside of this specific case, so we ask the following question. 
\begin{question}
Let $K$ be a field. 
Does there exist a $K$-algebra $A$ with a countable basis such that $A$ does not embed into $L_{2,K}$?
\end{question}
The authors have been unable to come up with any such $A$, but we would find an affirmative answer to the above question quite surprising. 

The second natural question to ask was originally asked of us by the anonymous referee: is $L_{2,R}$ the only Leavitt path algebra that admits embeddings of all other Leavitt path algebras of countable graphs? We thank the referee for this question, as well as suggesting Corollaries~\ref{cor: module type embed}~and~\ref{cor: pi simple embed} below.

The answer to this question is no. As an easy example note that $L_{2,R}$ embeds unitally into $L_{2,R} \oplus L_{2,R}$ as a $*$-algebra by the map $x \mapsto (x, x)$. Hence by Theorem~\ref{thm: Embed L(E) into L2} all Leavitt path algebras of countable graphs embed (unitally, when it makes sense) into $L_{2,R} \oplus L_{2,R}$ as $*$-algebras. 
We end this section with a few results about other Leavitt path algebras that admit embeddings of all Leavitt path algebras of countable graphs. 

\begin{lemma}
Let $K$ be a field and let $A$ be a unital $K$-algebra. 
Then $L_{2,K}$ embeds unitally into $A$ (as a $K$-algebra) if and only if $A$ has module type $(1,2)$. 
\end{lemma}
\begin{proof}
Recall (for instance from  \cite{LeavittOriginal}) that $A$ has module type $(1,2)$ if and only if there exists elements $x_1, x_2, y_1, y_2 \in A$ such that 
\[
  y_1 x_1 = 1 = y_2 x_2 \quad \text{ and } \quad x_1 y_1 + x_2 y_2 = 1.
\]
If such elements exist, then by the universal property of $L_{2,K}$ we can define a unital $K$-algebra homomorphism $\phi \colon L_{2,K} \to A$.
Since $L_{2,K}$ is simple it is necessarily an embedding. 
On the other hand, if $L_{2,K}$ embeds unitally into $A$, then the images of $a,a^*,b,$ and $b^*$ will show that $A$ has module type $(1,2)$. 
\end{proof}

\begin{corollary} \label{cor: module type embed}
Let $K$ be a field, and let $F$ be a graph with finitely many vertices. 
All Leavitt path algebras over countable graphs will embed (as $K$-algebras) into $L_K(F)$ if and only if $L_K(F)$ has module type $(1,2)$.
Furthermore, if $L_K(F)$ has module type $(1,2)$ then any unital Leavitt path algebra over a countable graph will embed unitally into $L_K(F)$.
\end{corollary}

\begin{corollary} \label{cor: pi simple embed}
Let $K$ be a field.
Let $F$ be a graph with finitely many vertices such that $L_{K}(E)$ is purely infinite simple and the class of the unit in $K_0(L_K(F))$ is $0$.
For any countable graph $E$ the Leavitt path algebra $L_{K}(E)$ embeds into $L_{K}(F)$ (as $K$-algebras), and if $E$ has finitely many vertices, then the embedding can be chosen unital. 
\end{corollary}
\begin{proof}
We say that two idempotents $e,f$ are equivalent if there exists elements $x,y$ such that $e = xy$ and $f = yx$, we denote this by $e \approx f$. (For the relation between $\sim$ and $\approx$ see Remark \ref{rem: proj equiv}.)

$L_K(F)$ is purely infinite, so we can find orthogonal idempotents $e,f \in L_{K}(f)$ such that $e \approx 1$, $f \neq 0$ and $1 = e + f$.
Since the class of $1$ in $K_0(L_K(F))$ is $0$, the class of $f$ is also $0$, so by \cite[Proposition 2.2]{AraGoodearlPardo} $f \approx 1$. 
Therefore we can find elements $x_1, x_2, y_1, y_2 \in L_{K}(F)$ such that 
\[
  x_1 y_1 = e, \quad x_2 y_2 = f, \quad \text{ and } \quad y_1 x_1 = 1 = y_2 x_2.
\]
That is $L_{K}(F)$ has module type $(1,2)$.
The conclusion now follows from Corollary \ref{cor: module type embed}.
\end{proof}

There are many algebras that satisfy the conditions in Corollary \ref{cor: pi simple embed}, for examples see for instance \cite[Section 4]{AbramsAnhLoulyPardo}.
We can even get algebras with non-trivial $K$-theory.

We note that Corollary \ref{cor: pi simple embed} only lets us conclude the existence of a $K$-algebra embedding, not, as in Theorem \ref{thm: Embed L(E) into L2} a $*$-algebra embedding.
The issue is that in definition of module type and $K$-theory for rings, we deal with equivalence of idempotents rather than projections.
However, in certain specific cases, we can directly prove that a $*$-algebra embedding exists. 

\begin{example}
Let $R$ be a unital commutative ring and let $F$ be the graph 

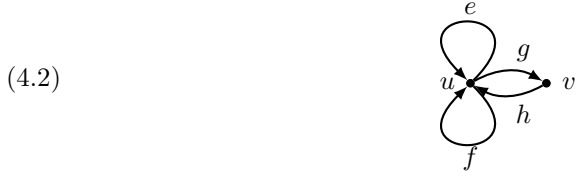
\vspace{-0.3cm}
\begin{equation}
\begin{tikzpicture}[baseline=(current  bounding  box.center)]

\node[circle, draw=black, fill=black, inner sep=1pt] (u) at (0,0) {};
\node[circle, draw=black, fill=black, inner sep=1pt] (v) at (1,0) {};
    
\draw[-latex,thick] (u) .. controls +(45:1.5cm) and +(135:1.5cm) .. (u);
\draw[-latex,thick] (u) .. controls +(315:1.5cm) and +(225:1.5cm) .. (u);
\draw[-latex,thick] (u) to [bend left] (v);
\draw[-latex,thick] (v) to [bend left] (u);

\node at (0,1) {$e$};
\node at (0,-1) {$f$};
\node at (0.7,0.4) {$g$};
\node at (0.7,-0.4) {$h$};

\node at (-0.3,0) {$u$};
\node at (1.3,0) {$v$};

\end{tikzpicture}
\end{equation}
We claim that for any countable graph $E$ the Leavitt path algebra $L_{R}(E)$ embeds into $L_{R}(F)$ (as $*$-algebras), and if $E$ has finitely many vertices, then the embedding can be chosen unital. 
To see this, it suffices to find a unital $*$-algebra embedding of $L_{2,R}$ into $L_{R}(F)$.
Let 
\[
  s = e + g, \quad t = f + v.
\]
Then we have that 
\begin{align*}
 s^*s &= (e^* + g^*)(e + g) = e^* e + g^*g = u + v = 1, \\
 ss^* &= (e + g)(e^* + g^*) = ee^* + gg^*, \\
 t^*t &= (f^* + v)(f + v) = f^*f + v = u + v = 1, \\
 t^*t &= (f + v)(f^* + v) = ff^* + v.
\end{align*}
By the universal property of $L_{2,R}$ we get a unital $*$-homomorphism $\phi \colon L_{2,R} \to L_R(F)$ such that $\phi(a) = s$, $\phi(b) = t$. 
By the Cuntz-Krieger Uniqueness Theorem (\cite[Theorem 6.5]{TomfordeLeavittOverRing}) $\phi$ is an embedding. 

Note that when $R$ is a principal ideal domain it follows from \cite[Corollary 7.7]{AraBrustengaCortinas} that $K_0(L_R(F)) = \Z/2\Z$, and in particular it is non-zero.
\end{example}

\section{Concrete embeddings}

The proof of Theorem \ref{thm: Embed L(E) into L2} is constructive, in that it gives a recipe for how to construct embeddings of $L_{R}(E)$ into $L_{2,R}$. 
In the case where $E$ satisfies Condition (L), the recipe is as follows:
\begin{enumerate}
 \item Pick orthogonal projections $\{p_v \mid v \in E^0\}$ such that $p_{v} \sim 1$ (Proposition \ref{prop: projs in L2}). 
 \item For each $v$, that isn't a sink, pick orthogonal projections $\{ q_{v,e} \mid e \in s^{-1}(v) \}$ such that $q_{v,e} \sim p_v$ and such that $p_v = \sum_{s(e) = v} q_{v,e}$ (Proposition \ref{prop: projs in L2}). 
 \item Pick partial isometries $t_e$ such that $t_e^* t_e = p_{r(e)}$ and $t_e t_e^* = q_{s(e),e}$ (such partial isometries exists since $q_{s(e),e} \sim p_{r(e)}$).
 \item Then $\{p_v, t_e\}$ is a Cuntz-Krieger $E$-family in $L_{2,R}$ and the $*$-homomorphism they define from $L_{R}(E)$ to $L_{2,R}$ is injective. 
\end{enumerate}

In the proof we use Proposition \ref{prop: projs in L2} to get the desired projections in $L_{2,R}$, but we can often make easier choices in concrete cases.
To help us do that we introduce the notion of a cylinder set. We denote the set of finite paths in the graph underlying $L_{2,R}$ by $\{a,b\}^*$, and the set of infinite paths by $\{a,b\}^\N$. For $\alpha\in \{a,b\}^*$ 
and 
$\xi\in\{a,b\}^\N$, $\alpha \xi\in\{a,b\}^\N$ denotes the obvious concatenation. For each path $\alpha \in \{a,b\}^*$ we define the cylinder set of $\alpha$, denoted $Z(\alpha)$, as 
\[
  Z(\alpha) = \{ \alpha \mu \mid \mu \in \{a,b\}^\N \} \subseteq \{a,b\}^\N,
\]
It is a well known consequence of relation (5) in Definition \ref{def: Leavitt path algebra over field} that if $\alpha_1, \alpha_2, \ldots, \alpha_n$ is a collection of paths with $\sqcup_i Z(\alpha_i) = \{a,b\}^\N$ then 
\[
 \sum_{i=1}^n \alpha_i \alpha_i^* = 1.
\]
And similarly if $\sqcup_i Z(\alpha_i) = Z(\beta)$ for some path $\beta$ then 
\[
 \sum_{i=1}^n \alpha_i \alpha_i^* = \beta \beta^*.
\]
If two paths $\alpha,\beta$ have disjoint cylinder sets then $\alpha^* \beta = 0$.

We can now describe a concrete embedding of finite graphs that satisfy Condition (L).

\begin{proposition} \label{prop: concrete finite embedding}
Let $E$ be a finite graph that satisfies Condition (L). 
Suppose we are given paths 
\begin{itemize}
 \item $\{ \alpha_v \in \{a,b\}^* \mid v \in E^0 \}$, and 
 \item $\{ \beta_{e} \in \{a,b\}^* \mid e \in E^1 \}$,
\end{itemize}
such that 
\begin{itemize}
 \item $\bigsqcup_{v \in E^0} Z(\alpha_v) = \{a,b\}^\N$, and 
 \item $\bigsqcup_{e \in s^{-1}(v)} Z(\beta_e) = Z(\alpha_v)$, for each $v$ that is not a sink. 
\end{itemize}
Then 
\[
 p_v = \alpha_v \alpha_v^*, \quad \quad \text{ and } \quad \quad t_e = \beta_e \alpha_{r(e)}^*
\]
form a Cuntz-Krieger $E$ family in $L_{2,R}$. 
Furthermore the $*$-homomorphism $\phi \colon L_{R}(E) \to L_{2,R}$ given by 
\[
 \phi(v) = p_v, \quad \quad \text{ and } \quad \quad \phi(e) = t_e, 
\]
is unital and injective. 
\end{proposition}
\begin{proof}
Since the cylinder sets $Z(\alpha_v)$ are disjoint the $p_v$ are pairwise orthogonal projections, and since $\alpha_v^* \alpha_v = 1$ we have that $p_v \sim 1$. 
For each $e \in E^1$, we let 
\[
 q_{s(e),e} = \beta_e \beta_e^*.
\]
Then $\beta_e$ witness the equivalence $q_{s(e),e} \sim 1$, so $q_{s(e),e} \sim p_{s(e)}$.
Since the cylinder sets $Z(\beta_e)$ are disjoint the $q_{v,e}$ are orthogonal projections and we have
\[
  \sum_{e \in s^{-1}(v)} q_{v,e} = p_v
\]
as
\[
  \bigsqcup_{e \in s^{-1}(v)} Z(\beta_e) = Z(\alpha_v).
\]
This shows that $p_v, q_{v,e}$ satisfies the first two points in the recipe. 

We now define 
\[
  t_e = \beta_e \alpha_{r(e)}^*.
\]
Then 
\[
  t_e^* t_e = \alpha_{r(e)} \beta_{e}^* \beta_{e} \alpha_{r(e)}^* = \alpha_{r(e)} \alpha_{r(e)}^* = p_{r(e)}, 
\]
and 
\[
  t_e t_e^* = \beta_{e} \alpha_{r(e)}^* \alpha_{r(e)} \beta_{e}^* = \beta_{e} \beta_{e}^* = q_{s(e),e}.
\]
Hence the $t_e$ satisfies the third point in the recipe. 
So by the fourth point $\{p_v, t_e\}$ is a Cuntz-Krieger $E$ family, and the $*$-homomorphism they define is injective.
\end{proof}

We use this to give concrete embeddings of some known Leavitt path algebras.

\begin{example}[Laurent polynomials]
We know from Proposition~\ref{prop: full spectrum} that $u = a a a^* + a b a^* b^* + b b^* b^*$ is a full spectrum unitary in $L_{2,R}$. As noted earlier, this means we have an embedding $L_R[z,z^{-1}]\hookrightarrow L_{2,R}$ mapping the polynomial $z$ to $u$.
\end{example}

\begin{example}[$L_{n,R}$]
Recall that $L_{n,R}$ is the Leavitt path algebra of the graph with one vertex and $n$ loops. We call the vertex $u$ and the loops $e_1, e_2, \ldots, e_n$. 
We wish to use Proposition \ref{prop: concrete finite embedding} to define an embedding. 
Since there is only one vertex we let $\alpha = \epsilon$ be the empty path. 
We now need to choose $n$ paths $\beta_1, \beta_2, \ldots, \beta_n$ such that $\sqcup_i Z(\beta_i) = \{a,b\}^\N$.
One way to do this is to put $\beta_i = a^{i-1}b$, for $i=1,2,\ldots,n-1$, and $\beta_n = a^{n-1}$. 
It now follows from Proposition~\ref{prop: concrete finite embedding} that the map $\phi \colon L_{n,R} \to L_{2,R}$ given on generators by 
\begin{align*}
      \phi(u) &= \alpha \alpha^* = 1, \\
      \phi(e_i) &= \beta_i = \begin{cases} a^{i-1}b, & i = 1,2,\ldots,n-1 \\ a^{n-1}, & i = n \end{cases}.
\end{align*}
is a unital $*$-homomorphic embedding. 
\end{example}

\begin{example}[The line graphs]
Let $A_n$ be the ``line graph'' with $n$ vertices. 
\vspace{-0.3cm}

\begin{equation}
\begin{tikzpicture}[baseline=(current  bounding  box.center)]

\node[circle, draw=black,fill=black, inner sep=1pt] (u1) at (0,0) {};
\node[circle, draw=black,fill=black, inner sep=1pt] (u2) at (1,0) {};
\node[circle, draw=black,fill=black, inner sep=1pt] (u3) at (2,0) {};
\node[circle, draw=black,fill=black, inner sep=1pt] (u4) at (3,0) {};
\node[circle, draw=black,fill=black, inner sep=1pt] (u5) at (4,0) {};
    
\draw[-latex,thick] (u1) -- (u2);
\draw[-latex,thick] (u2) -- (u3);    
\draw[thick, dotted] (u3) -- (u4);
\draw[-latex,thick] (u4) -- (u5);

\node at (0.2,0.3) {$u_1$};
\node at (1.2,0.3) {$u_2$};
\node at (2.2,0.3) {$u_3$};
\node at (3.2,0.3) {$u_{n-1}$};
\node at (4.2,0.3) {$u_n$};

\node at (0.5,-0.3) {$e_1$};
\node at (1.5,-0.3) {$e_2$};
\node at (3.5,-0.3) {$e_{n-1}$};

\end{tikzpicture}
\end{equation}

\vspace{-0.3cm}
Label the vertices and edges as above. 

We again wish to apply Proposition \ref{prop: concrete finite embedding}.
This time we first need to find $n$ paths $\alpha_1, \alpha_2, \ldots, \alpha_n$ such that $\sqcup_i Z(\alpha_i) = \{a,b\}^\N$.
Similar to the above example we define 
\[
  \alpha_i = \begin{cases} a^{i-1}b, & i = 1,2,\ldots,n-1 \\ a^{n-1}, & i = n. \end{cases}.
\]
For each $j = 1,2,\ldots, n-1$ we let $\beta_j = \alpha_i$. 
Then we are in a position to apply Proposition \ref{prop: concrete finite embedding}, which tells us that the map $\phi \colon L_R(A_n) \to L_{2,R}$ given on generators by
\begin{align*}
  \phi(u_i) &= \alpha_i \alpha_i^*,   \\
	\phi(e_j) &= \alpha_j \alpha_{j+1}^*,
\end{align*}
is an injective, unital $*$-homomorphism. 
\end{example}

\begin{example}[The Toeplitz algebra]
Let $\TT$ be the graph pictured below

\vspace{-0.3cm}

\begin{equation}
\begin{tikzpicture}[baseline=(current  bounding  box.center)]

\node[circle, draw=black,fill=black, inner sep=1pt] (u) at (0,0) {};
\node[circle, draw=black,fill=black, inner sep=1pt] (v) at (1,0) {};
    
\draw[-latex,thick] (u) .. controls +(120:1.5cm) and +(240:1.5cm) .. (u);
\draw[-latex,thick] (u) -- (v);
    
\node at (0.2,0.3) {$u$};
\node at (1.2,0.3) {$v$};
\node at (-0.8,0) {$e$};
\node at (0.5,-0.3) {$f$};

\end{tikzpicture}
\end{equation}

\vspace{-0.3cm}

We call the Leavitt path algebra $L_{R}(\TT)$ a Toeplitz algebra. 
To embed $L_{R}(\TT)$ into $L_{2,R}$ we define 
\[
	\alpha_1 = a, \quad \alpha_2 = b, \quad \beta_1 = aa, \quad \text{ and, } \quad \beta_2 = ab.
\]
Then $Z(\alpha_1) \sqcup Z(\alpha_2) = \{a,b\}^\N$ and $Z(\beta_1) \sqcup Z(\beta_2) = Z(\alpha_1)$, so by Proposition \ref{prop: concrete finite embedding} the map $\phi \colon L_{R}(\TT) \to L_{2,R}$ given on generators by 
\[
	\phi(u) = aa^*, \quad \phi(v) = bb^*,\quad \phi(e) = aaa^* \quad \text{ and } \quad \phi(f) = abb^*,
\] 
is a unital $*$-homomorphic embedding. 
\end{example}

\section*{Acknowledgements}

The authors thank Efren Ruiz for pointing out the reference \cite{AbramsRangaswamy} to them. 
The second-named author is grateful to Wojciech Szyma{\'n}ski for 
enlightening 
conversations about endomorphisms of $\OO_2$. 


The authors thank Gene Abrams, Pere Ara, and Enrique Pardo for helpful comments on an earlier version of this paper that helped improve it. 

During work on this project the first named author visited the University of 
Oslo. The first named author would like to thank the Institute for Mathematics 
\& Its Applications at the University of Wollongong for providing financial 
support, and the University of Oslo and Nadia Larsen for also providing 
financial 
support, and for their hospitality. 

Part of this work was done while the second named author was supported by an individual post doctoral grant from the Danish Council for Independent Research \textbar {} Natural Sciences. 

\bibliographystyle{plain}


\end{document}